\documentclass[12pt]{article}

\usepackage{amsmath,amsthm,amsfonts,amssymb,amscd}

\usepackage{epsfig}
\usepackage{graphicx}
\usepackage{multirow}
\usepackage{array}
\usepackage[all]{xypic}

\newtheorem{theorem}{Theorem}[section]

\newtheorem{lemma}[theorem]{Lemma}
\newtheorem{remark}[theorem]{Remark}

\def\ml{\mathcal{C}}
\def\ml1{\mathcal{C}^1}
\def\mlb1{\mathcal{C}_{b}^{1}}

\def\mc{\mathbb{C}}
\def\mz{\mathbb{Z}}
\def\PP{\mathbb{P}}

\usepackage{color}
\usepackage{ifthen}

\newboolean{usecolor}
\setboolean{usecolor}{true}
\ifthenelse{\boolean{usecolor}}
{
  \definecolor{colore}{cmyk}{0,1,0.6,0}
  \definecolor{coloregen}{cmyk}{0.7,0,1,0}
  \definecolor{coloresimo}{cmyk}{1,0.6,0,0}
}
{
  \definecolor{colore}{cmyk}{0,0,0,1}
  \definecolor{coloregen}{cmyk}{0,0,0,1}
  \definecolor{coloresimo}{cmyk}{0,0,0,1}
}

\title{Braid groups in complex Grassmannians}
\author{Sandro {\sc Manfredini}\footnote{Department of Mathematics, University of Pisa. manfredi@dm.unipi.it}  \and Simona {\sc Settepanella}\footnote{Department of Mathematics, Hokkaido University. s.settepanella@math.sci.hokudai.ac.jp}}
\begin{document}

\maketitle

\begin{abstract}
We describe the fundamental group and second homotopy group of ordered $k-$point sets
in $Gr(k,n)$ generating a subspace of fixed dimension. 
\end{abstract}

\begin{center}
{\small\noindent{\bf Keywords}:\\
complex space, configuration spaces, \\braid groups.}
\end{center}

\begin{center}
\end{center}

\section{Introduction}

Let $M$ be a manifold and $\Sigma_h$ be the symmetric group on $h$ elements.\\
 The \emph{ordered} and \emph{unordered configuration spaces} of $h$ distinct points in $M$, $\mathcal{F}_h(M)=\{(x_1,\ldots,x_h)\in
M^h|x_i\neq x_j,\,\,i\neq j\}$ and $\mathcal{C}_h(M)=\mathcal{F}_h(M)/\Sigma_h$, have been widely studied. In recent papers (\cite{BP,MPS,MS}), new configuration spaces were introduced when $M$ is, respectively, the projective space $\mc\PP^n$, the affine space $\mc^n$ and the Grassmannian manifold $Gr(k,n)$ of $k$-dimensional subspaces of $\mc^n$, by stratifying the configuration spaces $\mathcal{F}_h(M)$ (resp. $\mathcal{C}_k(M)$) with complex submanifolds $\mathcal{F}_h^i(M)$ (resp. $ \mathcal{C}_h^i(M)$)
defined as the ordered (resp. unordered) configuration spaces of all $h$ points in $M$ generating a subspace of dimension $i$.
The homotopy groups of those configuration spaces are interesting as they are strongly related to the homotopy groups of the Grassmannian manifolds, i.e. of spheres.\\
In \cite{BP} (resp. \cite{MPS}), the fundamental groups $\pi_1(\mathcal{F}_h^i(\mc\PP^n))$ and $\pi_1(\mathcal{C}_h^i(\mc\PP^n))$ (resp. $\pi_1(\mathcal{F}_h^i(\mc^n))$ and $\pi_1(\mathcal{C}_h^i(\mc^n))$) are computed, proving that the former are trivial and the latter are isomorphic to the symmetric group $\Sigma_h$ except when $i=1$ (resp. $i=1$ and $i=n=h-1$) providing, in this last case, a presentation for both $\pi_1(\mathcal{F}_h^1(\mc\PP^n))$ and $\pi_1(\mathcal{C}_h^1(\mc\PP^n))$ (resp. $\pi_1(\mathcal{F}_h^i(\mc^n))$ and $\pi_1(\mathcal{C}_h^i(\mc^n))$) which is similar to those of the braid groups of the sphere.\\ 
In this paper we generalize the results obtained in \cite{BP} when $M$ is the projective space $\mc\PP^{n-1}=Gr(1,n)$, to the case of Grassmannian manifold $Gr(k,n)$ of $k$-dimensional subspaces of $\mc^n$. We prove that if $\mathcal{F}_h^i(k,n)$ is the $i$-th ordered configuration space of all distinct points $H_1,\ldots,H_h$ in the Grassmannian manifold $Gr(k,n)$ whose sum is a subspace of dimension $i$, then the following result holds.

\begin{theorem}\label{teo:prin} The $i$-th ordered configuration spaces $\mathcal{F}_h^i(k,n)$ are all simply connected if $k>1$. 
\end{theorem}

From this, we immediately obtain that the fundamental group of the $i$-th unordered configuration space $\mathcal{F}_h^i(k,n)/\Sigma_h$ is isomorphic to $\Sigma_h$.\\
These results are stated in Section \ref{s:uno}. In Section \ref{s:due} we compute the second homotopy group of the $i$-th configuration spaces in two special cases: the case in which the subspaces are in direct sum and the case of two subspaces.

\section{The first homotopy group of $\mathcal{F}_h^i(k,n)$} \label{s:uno}

Let $Gr(k,n)$ be the grassmannian manifold pa\-ra\-me\-tri\-zing $k$-dimensional subspaces of $\mc^n$, $0<k<n$. 
In \cite{MS} authors define the space $\mathcal{F}_h^i(k,n)$ as the ordered configuration space of all $h$
distinct points $H_1,\ldots,H_h$ in $Gr(k,n)$ such that the dimension of the sum dim$(H_1+\cdots+H_h)$ equals $i$.

\begin{remark}\label{rm:1} The following easy facts hold:
\begin{enumerate}
\item if $h=1$, $\mathcal{F}_h^i(k,n)$ is empty except for $i=k$ and $\mathcal{F}_1^k(k,n)=Gr(k,n)$;
\item if $i=1$, $\mathcal{F}_h^i(k,n)$ is empty except for $k,h=1$ and $\mathcal{F}_1^1(1,n)=Gr(1,n)=\mc\PP^{n-1}$;
\item if $h\geq 2$ and $k=n-1$ then $\mathcal{F}_h^i(k,n)$ is empty except for $i=n$, and, since the sum of two (different) hyperplanes is $\mc^n$, $\mathcal{F}_h^n(n-1,n)=\mathcal{F}_h(Gr(n-1,n))=\mathcal{F}_h(\mc\PP^{n-1})$;
\item if $h\geq 2$ then ${\mathcal{F}_h^i(k,n)}\neq \emptyset$ if and only if $i\geq k+1$ and $i\leq
\min(kh,n)$; 
%\item for $i=hk\leq n$, then the $h$ subspaces giving a point of $\mathcal{F}_h^{hk}(k,n)$ form a direct sum;
\item if $h\geq 2$ then $\mathcal{F}_h(Gr(k,n))=\mathop{\coprod}\limits_{i=k+1}^{{\rm min}(hk,n)} \mathcal{F}_h^i(k,n)$,
with the open stratum given by the case of maximum dimension $i=$min$(hk,n)$;
\item if $h\geq 2$ then the adjacency of the strata is given by
$$
\overline{\mathcal{F}_h^i(k,n)}=\mathcal{F}_h^{k+1}(k,n)\coprod\ldots\coprod\mathcal{F}_h^i(k,n).
$$
\end{enumerate}
\end{remark}

As the case $k=1$ has been treated in \cite{BP} and, by the above remarks, the case $h=1$ is trivial, in this paper we will consider $h,k>1$ (and hence $i>k$).\\

In \cite{MS}, authors proved that $\mathcal{F}_h^i(k,n)$ is (when non empty) a complex sub\-ma\-ni\-fold of $Gr(k,n)^h$ of dimension $i(n-i)+hk(i-k)$, and that if  $i=$min$(n,hk)$ and $n\neq hk$ then the open strata $\mathcal{F}_h^i(k,n)$ are simply connected except for $n=2$ (and $k=1$), i.e. 
\begin{equation}\label{casodiretto}
\pi_1(\mathcal{F}_h^{{\rm min}(n,kh)}(k,n))=\left \{ \begin{split}& 0 &{\rm if}\ n\neq hk \\ & \mathcal{PB}_h(S^2) & {\rm if}\ n=2,\ k=1 
\end{split} \right .
\end{equation}
where $\mathcal{PB}_h(S^2)$ is the pure braid group on $h$ strings of the sphere $S^2$.

In order to complete this result and compute fundamental groups in all cases we need two Lemmas.

\begin{lemma}\label{pr:2} Let $V=(H_1,\ldots,H_h)$ be an element in the space $\mathcal{F}_h^i(k,n)$ and denote the sum $H_1+\cdots+H_h\in Gr(i,n)$ by $\gamma(V)$, then the map
\begin{equation}\label{map:gamma}
\gamma: \mathcal{F}_h^i(k,n)  \rightarrow  Gr(i,n)
\end{equation}
is a locally trivial fibration with fiber $\mathcal{F}_h^i(k,i)$.
\end{lemma}

\begin{proof}Let $V_0$ be an element in the Grassmannian manifold $Gr(i,n)$. Fix 
$L_0\in Gr(n-i,n)$ such that $L_0\oplus V_0=\mc^n$ and let $\varphi:\mc^n\to V_0$ be the linear projection on $V_0$ given by the direct sum.
If $\mathcal{F}_h^i(k,V_0)$ is the ordered configuration space of $h$ distinct $k$-dimensional spaces in $V_0$ whose sum is an $i$-dimensional subspace, then $\mathcal{F}_h^i(k,V_0)$ coincides with $\mathcal{F}_h^i(k,i)$ when a basis in $V_0$ is fixed.

Let $\mathcal{U}_{L_0}$ be the open neighborhood of $V_0$ in $Gr(i,n)$ defined as
$$
\mathcal{U}_{L_0}=\{V\in Gr(i,n)|\ L_0\oplus V=\mc^n\}.
$$
The restriction of the projection $\varphi$ to an element $V$ in $\mathcal{U}_{L_0}$ is a linear isomorphism $\varphi_V:V\rightarrow V_0$ and a local trivialization for $\gamma$ is given by the homeomorphism
\begin{equation*}
\begin{split}
 f:\gamma^{-1}(\mathcal{U}_{L_0}) &\rightarrow \mathcal{U}_{L_0}\times \mathcal{F}_h^i(k,V_0)\\
 y=(H_1,\ldots,H_h)  &\mapsto \big(\gamma(y),(\varphi_{\gamma(y)}(H_1),\ldots,\varphi_{\gamma(y)}(H_h))\big)
\end{split}
\end{equation*}
which makes the following diagram commute.
	
\begin{center}
\begin{picture}(360,120)
\thicklines
%\put(10,10){\line(1,0){100}}
\put(77,90){\vector(2,-1){45}}
\put(107,100){\vector(1,0){70}}
\put(197,90){\vector(-2,-1){45}}
\put(55,98){${\gamma}^{-1}(\mathcal{U}_{L_0})$}
\put(184,98){$\mathcal{U}_{L_0}\times \mathcal{F}_h^i(k,i)$}
\put(130,60){$ \mathcal{U}_{L_0}$}
\put(83,70){$\gamma$}
\put(180,70){$pr_1$}
\put(134,105){$f$}
%\put(350,60)
\end{picture}
\end{center}
\vspace{-1.7cm}
This completes the proof. \end{proof}

\bigskip

\begin{lemma}\label{lem:pr}
\label{pr:2} The projection map on the first $h-1$ entries
\begin{equation}\label{map:proj}
\begin{split}
pr:\mathcal{F}_h^{kh}(k,kh) &\rightarrow \mathcal{F}_{h-1}^{k(h-1)}(k,kh)\\
(H_1,\ldots,H_h) &\mapsto (H_1,\ldots, H_{h-1})
\end{split}
\end{equation}
is a locally trivial fibration with fiber $\mc^{k(kh-k)}$.
\end{lemma}

\begin{proof} Let $V_0$ be an element in $\mathcal{F}_{h-1}^{k(h-1)}(k,kh)$. Fix $L_0\in Gr(k,kh)$ such that $L_0\oplus \gamma(V_0)=\mc^n$ and let $\varphi:\mc^n\to \gamma( V_0)$ be the linear projection on $\gamma(V_0)$ given by the direct sum.\\
The fiber of the projection map $pr$ over $ V_0$ is the open set $$U_{\gamma( V_0)}=\{H\in Gr(k,kh)|H\oplus \gamma( V_0)=\mc^n\}$$ 
which is homeomorphic to $\mc^{k(kh-k)}$ (it is a coordinate chart for the Grassmannian manifold $Gr(k,hk)$).

Let $\mathcal{U}_{L_0}$ be the open neighborhood of $ V_0$ in $\mathcal{F}_{h-1}^{k(h-1)}(k,kh)$ defined as
$$
\mathcal{U}_{L_0}=\{V\in \mathcal{F}_{h-1}^{k(h-1)}(k,kh)|\ L_0\oplus \gamma(V)=\mc^n\}.
$$
If $V$ is a point in $\mathcal{U}_{L_0}$, the restriction of the map $\varphi$ to $\gamma(V)$ is a linear isomorphism $\tilde\varphi_V:\gamma(V)\rightarrow \gamma( V_0)$ that can be extended to an isomorphism $\varphi_V$ of $\mc^n$ by requiring it to be the identity on $L_0$.\\
A local trivialization for the projection $pr$ is given by the homeomorphism
\begin{equation*}
\begin{split}
f:pr^{-1}(\mathcal{U}_{L_0})&\rightarrow \mathcal{U}_{L_0}\times U_{\gamma( V_0)}\\
y=(H_1,\ldots,H_h)&\mapsto\big(pr(y),\varphi_{\gamma(pr(y))}(H_h)\big)
\end{split}
\end{equation*}
which makes the following diagram commute. 
\begin{center}
\begin{picture}(360,120)
\thicklines
%\put(10,10){\line(1,0){100}}
\put(77,90){\vector(2,-1){45}}
\put(107,100){\vector(1,0){70}}
\put(197,90){\vector(-2,-1){45}}
\put(55,98){${pr}^{-1}(\mathcal{U}_{L_0})$}
\put(184,98){$\mathcal{U}_{L_0}\times U_{\gamma(V_0)}$}
\put(130,60){$ \mathcal{U}_{L_0}$}
\put(83,70){$pr$}
\put(180,70){$pr_1$}
\put(134,105){$f$} 
%\put(350,60){$\square$}
\end{picture}
\end{center}
\vspace{-1.7cm}
This completes the proof. \end{proof}

\bigskip
 
Let us remark that if $V=(H_1,\ldots,H_h)$ is a point in the space $\mathcal{F}_h^{kh}(k,n)$, i.e. $i=kh$, then the $h$ subspaces $H_1,\ldots, H_h$ are in direct sum and the map 
\begin{equation*}
\begin{split}
pr:\mathcal{F}_h^{kh}(k,n)&\rightarrow \mathcal{F}_{h-1}^{k(h-1)}(k,n)\\
(H_1,\ldots,H_h)&\mapsto (H_1,\ldots, H_{h-1})
\end{split}
\end{equation*}
is well defined. With a similar argument to the one used in Lemma \ref{lem:pr} it can be proved that this map is a locally trivial fibration even if $n\not=hk$. In this case the fiber (which can be described as the union of coordinate charts $U_W$ of the Grassmannian manifold) is still simply connected but not homotopically trivial.

We have, from the homotopy long exact sequence of the fibration $pr$, that
\begin{equation}\label{om:pr}
\pi_j(\mathcal{F}_h^{kh}(k,kh))=\pi_j(\mathcal{F}_{h-1}^{k(h-1)}(k,kh))
\end{equation}
for all $j$ and, by equation (\ref{casodiretto}), that 
$$
\pi_1(\mathcal{F}_h^{kh}(k,kh))=\pi_1(\mathcal{F}_{h-1}^{k(h-1)}(k,kh))=0.
$$
It follows that the open stratum $\mathcal{F}_h^{kh}(k,kh)$ is simply connected, hence all open strata are simply connected.\\
Moreover, from the homotopy long exact sequence of the fibration $\gamma$, we have that
$$
\pi_1(\mathcal{F}_h^i(k,i))\to\pi_1(\mathcal{F}_h^i(k,n))\to\pi_1(Gr(i,n))=0.
$$
As $\mathcal{F}_h^i(k,i)$ is an open stratum, it is simply connected and hence $\pi_1(\mathcal{F}_h^i(k,n))=0$.

That is all our configuration spaces are simply connected and Theorem \ref{teo:prin} is proved.

\section{The second homotopy group}\label{s:due}
In this section we compute the second homotopy group $\pi_2(\mathcal{F}_h^i(k,n))$ when $i=hk$, i.e. subspaces in direct sum, and when $h=2$, i.e. the case of two subspaces.
In order to compute those homotopy groups, we need to know that the third homotopy group for Grassmannian manifolds is trivial if $k>1$. Even if it should be a classical result we didn't find references and we decided to give a proof here. 

Let $V_{k,n}$ be the space parametrizing the (ordered) $k$-uples of orthonormal vectors in $\mc^n$, $1\leq k\leq n$. It is an easy remark that $V_{1,n}=S^{2n-1}$ and $V_{n,n}=U(n)$.
It's well known that the function that maps an element of $V_{k,n}$ to the subspace generated by its entries is a locally trivial fibration:
\begin{equation}\label{fibr:inc}
V_{k,k}\hookrightarrow V_{k,n}\to Gr(k,n)\ \ (k<n) ,
\end{equation}
while the projection on the last entry is the locally trivial fibration: 
\begin{equation}\label{fibr:del}
V_{k-1,n-1}\hookrightarrow V_{k,n}\to S^{2n-1}\ \ (k>1) .
\end{equation}

The fibration in (\ref{fibr:del}) induces the long exact sequence in homotopy
 \begin{equation}\label{hom}
\begin{split}
\to &\pi_4(S^{2n-1})\to\pi_3(V_{k-1,n-1})\to \pi_3(V_{k,n})\to \pi_3(S^{2n-1})\to\pi_2(V_{k-1,n-1})\\
&\to \pi_2(V_{k,n})\to \pi_2(S^{2n-1})\to\pi_1(V_{k-1,n-1})\to \pi_1(V_{k,n}) \to \pi_1(S^{2n-1})\to 0
\end{split}
\end{equation}

If $n=k=2$, since $V_{1,1}=S^1$, the sequence in (\ref{hom}) becomes:
\begin{equation*}
\begin{split}
0\to\pi_3(&V_{2,2})\to\mz\to 0\to \\
&\to \pi_2(V_{2,2})\to 0\to\mz\to \pi_1(V_{2,2})\to 0  \quad ,
\end{split}
\end{equation*}
that is $\pi_1(V_{2,2})=\mz$, $\pi_2(V_{2,2})=0$, $\pi_3(V_{2,2})=\mz$.\\ 
If $n>2$ by (\ref{hom}) we get isomorphisms $\pi_i(V_{k,n}) \simeq \pi_i(V_{k-1,n-1})$, $1~\leq~i~\leq~3$. This implies that,
if $k=n$ then $\pi_i(V_{n,n}) \simeq \pi_i(V_{n-1,n-1}) \simeq \pi_i(V_{2,2})$, thus $\pi_1(V_{n,n})=\mz$, $\pi_2(V_{n,n})=0$, $\pi_3(V_{n,n})=\mz$. Otherwise $\pi_i(V_{k,n}) \simeq \pi_i(V_{k-1,n-1})\simeq \pi_1(V_{1,n-k+1})$ and hence $\pi_1(V_{k,n})=0$, $\pi_2(V_{k,n})=0$, $\pi_3(V_{k,n})=0$ except if $k=n-1$ in which case  $\pi_3(V_{n-1,n})=\mz$.\\

By the above considerations, the exact sequence of homotopy groups associated to the fibration (\ref{fibr:inc}) for $k<n-1$ becomes 
\begin{equation*}
\begin{split}
\mz\to 0\to\pi_3(Gr(k,n))\to 0\to &0\to \pi_2(Gr(k,n))\to \\
& \to \mz\to 0\to \pi_1(Gr(k,n))\to 0 \quad ,
\end{split}
\end{equation*}
that is $\pi_1(Gr(k,n))=0$, $\pi_2(Gr(k,n))=\mz$ and $\pi_3(Gr(k,n))=0$ if $k<n-1$.\\
If $k=n-1$ then $Gr(n-1,n)=\PP^{n-1}$ and $\pi_3(Gr(n-1,n))=0$ except if $n=2$ in which case $Gr(1,2)=S^2$ and $\pi_3(Gr(1,2))=\mz$.
That is the third homotopy group of the Grasmannian manifold $Gr(k,n)$ is trivial if $k>1$.\\

Since the third homotopy group of the Grasmannian manifold $Gr(k,n)$ is trivial if $k>1$ then for $i<n$ the homotopy long exact sequence of the fibration $\gamma$ defined in equation (\ref{map:gamma}) gives :
$$
0=\pi_3(Gr(i,n))\to\pi_2(\mathcal{F}_h^i(k,i))\to\pi_2(\mathcal{F}_h^i(k,n))\to\mz=\pi_2(Gr(i,n))\to0.
$$
As the second homotopy groups are abelian and the above short exact sequence splits, we have
$$
\pi_2(\mathcal{F}_h^i(k,n))=\pi_2(\mathcal{F}_h^i(k,i))\times\mz.
$$ 
\paragraph{The case $i=hk$.} If $i=hk$, by equation (\ref{om:pr}), $\pi_2(\mathcal{F}_h^{hk}(k,hk))=\pi_2(\mathcal{F}_{h-1}^{k(h-1)}(k,hk))$
and the following equalities hold:
\begin{eqnarray*}
\pi_2(\mathcal{F}_h^{hk}(k,hk))&=&\pi_2(\mathcal{F}_{h-1}^{k(h-1)}(k,k(h-1)))\times\mz=\\
&=&\pi_2(\mathcal{F}_{h-2}^{k(h-2)}(k,k(h-1)))\times\mz=\\
&=&\pi_2(\mathcal{F}_{h-2}^{k(h-2)}(k,k(h-2)))\times\mz^2=\\
&=&\pi_2(\mathcal{F}_{2}^{2k}(k,2k))\times\mz^{h-2}=\\
&=&\pi_2(\mathcal{F}_{1}^{k}(k,2k))\times\mz^{h-2}=\\
&=&\pi_2(Gr(k,2k))\times\mz^{h-2}=\\
&=&\mz^{h-1}\\
\end{eqnarray*}
while, if $hk<n$, $\pi_2(\mathcal{F}_h^{hk}(k,n))=\mz^h$.

\paragraph{The case $h=2$.} If $h=2$ a point $(H_1,H_2)$ is in the space $\mathcal{F}_2^i(k,n)$ if and only if the dimension of intersection dim$(H_1\cap H_2)=2k-i$.
If $i=2k$ (which includes the cases $k=1$ and $n=2$) $H_1$ and $H_2$ are in direct sum otherwise the following Lemma holds.
\begin{lemma}
\label{pr:2} If $i<2k$, the map
\begin{equation*}
\begin{split}
\eta:\mathcal{F}_2^i(k,n)&\rightarrow Gr(2k-i,n)\\
(H_1,H_2) &\mapsto H_1\cap H_2
\end{split}
\end{equation*}
is a locally trivial fibration with fiber $\mathcal{F}_2^{2i-2k}(i-k,n-2k+i)$.
\end{lemma}

\begin{proof} Let $V_0$ be a point in the Grassmannian manifold $Gr(2k-i,n)$. Fix $L_0\in Gr(n-2k+i,n)$ such that $L_0\oplus V_0=\mc^n$ and let $\varphi:\mc^n\to V_0$ be the linear projection given by the direct sum.\\
The fiber $\eta^{-1}(V_0)$ is the set of all pairs $(H_1,H_2)$ of $k$-dimensional subspaces of $\mc^n$ such that $H_1\cap H_2=V_0$.  
That is, a pair $(H_1,H_2)$ is in $\eta^{-1}(V_0)$ if and only if it corresponds to a pair of $(i-k)$-dimensional subspaces of $\mc^n/V_0$ are in direct sum, i.e. a point in $\mathcal{F}_2^{2(i-k)}(i-k,n-2k+i)$.

Let $\mathcal{U}_{L_0}$ be the open neighborhood of $V_0$ in $Gr(2k-i,n)$, defined as
$$
\mathcal{U}_{L_0}=\{V\in Gr(2k-i,n)|\ L_0\oplus V=\mc^n\}.
$$
If $V$ is a point in $\mathcal{U}_{L_0}$, the restriction of $\varphi$ to $\gamma(V)$ is a linear isomorphism $\tilde\varphi_V:V\rightarrow  V_0$ that can be extended to an isomorphism $\varphi_V$ of $\mc^n$ by requiring it to be the identity on $L_0$.

A local trivialization for $\eta$ is the homeomorphism
\begin{equation*}
\begin{split}
f:\eta^{-1}(\mathcal{U}_{L_0})&\rightarrow \mathcal{U}_{L_0}\times \eta^{-1}(V_0)\\
(H_1,H_2)&\mapsto\big(\eta(y),(\varphi_{\eta(y)}(H_1),\varphi_{\eta(y)}(H_2))\big)
\end{split}
\end{equation*}
 This completes the proof. \end{proof}

\bigskip

By the homotopy long exact sequence of the map $\eta$, we get:
$$
0\to\pi_2(\mathcal{F}_2^{2i-2k}(i-k,n-2k+i))\to\pi_2(\mathcal{F}_2^{i}(k,n))\to\mz\to0
$$
and hence $\pi_2(\mathcal{F}_2^{i}(k,n))=\mz\times\pi_2(\mathcal{F}_2^{2(i-k)}(i-k,n-2k+i))$. By the previous case, $\pi_2(\mathcal{F}_2^{2(i-k)}(i-k,n-2k+i))$ is equal to $\mz$ if $2(i-k)=n-2k+i$, that is if $i=n$, and is equal to $\mz^2$ otherwise. So,
we get $\pi_2(\mathcal{F}_2^{n}(k,n))=\mz^2$ and $\pi_2(\mathcal{F}_2^{i}(k,n))=\mz^3$ if $i<n$.

\end{document}